\documentclass[11pt]{article}

\usepackage[a4paper,margin=2.5cm]{geometry}
\usepackage{amsmath,amssymb,amsthm,mathtools}
\usepackage{booktabs}
\usepackage{enumitem}
\usepackage{microtype}
\usepackage{placeins}
\usepackage{hyperref}
\hypersetup{hidelinks}
\usepackage{tikz}
\usepackage{pgfplots}
\usepackage[most]{tcolorbox}
\usetikzlibrary{arrows.meta,positioning,calc,fit}
\pgfplotsset{compat=1.18}

\newtcolorbox{architectnote}[1][]{
  enhanced,breakable,
  colback=black!2,colframe=black!55,
  boxrule=0.5pt,arc=1.2mm,
  left=2mm,right=2mm,top=1.5mm,bottom=1.5mm,
  fonttitle=\bfseries,
  title={Structural note},#1
}
\newtcolorbox{challengebox}[1][]{
  enhanced,breakable,
  colback=black!1,colframe=black!75,
  boxrule=0.8pt,arc=1.2mm,
  left=2mm,right=2mm,top=1.5mm,bottom=1.5mm,
  fonttitle=\bfseries,
  title={Unlabeled challenge},#1
}

\newtheorem{proposition}{Proposition}
\theoremstyle{definition}

\newcommand{\R}{\mathbb{R}}

\newcommand{\Lie}{\mathcal{L}}

\title{Exactness as an Explanatory and Computational Target:\\
A Reduction Architecture for First-Order ODEs, with a Higher-Order Outlook}
\author{Gabriel Ben-Simon\\
\small Afeka Academic College of Engineering\\
\small Tel Aviv, Israel\\
\texttt{GavrielBe@afeka.ac.il}}
\date{}
\hypersetup{pdftitle={Exactness as an Explanatory and Computational Target: A Reduction Architecture for First-Order ODEs, with a Higher-Order Outlook},pdfauthor={Gabriel Ben-Simon}}

\begin{document}
\maketitle

\begin{abstract}
A typical first-order ODE syllabus follows a familiar sequence of topics, which can leave students with little sense of unity among the methods. This ``collection of tricks'' problem is well known to instructors. We propose organizing a substantial part of the classical symbolic first-order syllabus around a single computational target: an exact representative. The fundamental ``atomic'' ODE, \(y'=g(x)\), is solved directly because
\[
d\left(y-\int^x g(s)\,ds\right)=0.
\]
More generally, represent the equation by a one-form on the plane, \(\omega=M\,dx+N\,dy=0\). If \(\omega\) is locally closed, it is locally exact, and one integration produces a first integral. If it is not closed, intermediate actions are needed. The principal symbolic reductions considered here use two recurring structural actions: coordinate adaptation and nonvanishing rescaling by an integrating factor. Thus the reduction problem is to find, on an appropriate coordinate patch, \(\Phi\) and a nonvanishing \(\mu\) such that
\[
\mu\Phi^*\omega=dG,
\]
so that the solution curves are level sets \(G=C\). These ingredients are classical; the contribution here is their target-centered organization. This also yields a two-floor pedagogical view: the \textbf{structural floor} records the target, transformations, and domain information, while the \textbf{operational floor} executes substitutions, normalizations, and integrations. In an AI- and CAS-assisted environment, this suggests emphasizing recognition, justification, auditing, and interpretation while selected routine symbolic steps are delegated. We finally indicate how Lie symmetry systematizes coordinate adaptation and how constant-coefficient linear systems echo the same reduction logic.

This extended version also develops reverse construction, a structural exercise laboratory, comparisons of reduction paths, and a higher-order outlook through first integrals, operator factorization, and structural normalization.
\end{abstract}

\section{Introduction: From direct integration to a master reduction problem}

A first course in ordinary differential equations usually combines three broad themes: existence, uniqueness, and techniques of solution. From the student's point of view, the algorithmic side often receives the most visible emphasis. The techniques are then presented through a succession of named solvable classes---separable, linear, exact, homogeneous, Bernoulli, and, in some courses, Riccati equations. This organization is efficient for manual practice (recognize a pattern, recall the method, execute), but it can leave the mathematics feeling disjointed.

The viewpoint developed here grew out of more than a decade of repeatedly teaching elementary ODEs, together with courses in analysis and linear algebra. The recurring question was whether the standard computational classes could be presented as parts of as few mathematical mechanisms as possible rather than as isolated methods.

Our paper has an important AI angle. Artificial Intelligence (AI) and Computer Algebra Systems (CAS) can now carry out many routine symbolic steps that once occupied a large fraction of hand calculation. In that computational environment, a bag-of-tricks organization \cite{Yap2010} is no longer a sufficient intellectual story by itself. So one may ask: what mathematical structure should the student recognize, justify, and control even when selected calculations are delegated? In the language developed below, the student is asked to act as a \emph{structural architect}: choosing admissible transformations, identifying lost branches (for example after division), and guiding the equation toward a representation in which direct integration becomes available. To start answering, we ask:
What makes the simplest differential equation, \(y'=g(x)\), so elementary to solve? It requires no classification because it is already in the \emph{atomic reference state} of direct integration. Equivalently, its differential one-form is exact:
\[
dy-g(x)\,dx = d\left(y-\int^x g(s)\,ds\right)=0.
\]

Our main observation is that a substantial part of the classical symbolic first-order syllabus can be read as a quest to return to this atomic state. An ODE represented by a one-form $\omega=Mdx+Ndy=0$ is solved by one operation of integration if it is locally closed and thus exact. If it is not, we need intermediate actions to turn it into one. We claim here that the principal symbolic methods of a basic first-order ODE course can be organized around two essential actions.
\[
\boldsymbol{\text{find an admissible }\Phi\text{ and a nonvanishing }\mu\text{ such that }\mu\Phi^*\omega=dG.}
\]
Equivalently, one seeks a closed representative \(d(\mu\Phi^*\omega)=0\) on a coordinate patch where closedness yields a local potential. Once \(G\) is available, the solution is the level set \(G=C\). 

This perspective naturally separates the curriculum into two interacting levels:
\begin{itemize}
    \item \textbf{The Structural Floor:} The mathematical architecture. Here we define the target, select the pullback \(\Phi^*\) and the rescaling \(\mu\), and control domains and excluded branches.
    \item \textbf{The Operational Floor:} The algorithmic execution of the named methods, unpackaging standard substitutions and integration steps.
\end{itemize}

None of the mathematical ingredients presented here are newly discovered; Frobenius's theorem, exact forms, and Lie symmetries are well-established classical concepts \cite{Arnold1992, Hydon2000, Olver1993, Teschl2012}. The historical dots are already there; the contribution proposed here is the particular target-centered architecture obtained by connecting them.

This extended version retains that core first-order architecture and builds on it through reverse construction from \(G,\Phi,\mu\), a structural exercise laboratory, comparisons of reduction paths, and a higher-order outlook. These additions do not alter the first-order claim that exactness is the terminal computational target: once an exact representative \(dG\) has been reached, a first-order equation has already been reduced to the level-set relation \(G=C\), so there is no lower differential order left to reach. The later sections extend what the architecture can be used to organize; they do not replace its terminal first-order node.

The order of the core sections reflects this logic. Section~2 fixes the target and separates the two structural moves; Section~3 asks which familiar classes arise when the rescaling is restricted to depend on only one coordinate; Section~4 assembles the standard named classes into a reduction hierarchy; and Section~5 addresses the discovery problem---where a useful coordinate adaptation can come from when it is not algebraically obvious---through Lie symmetry. Sections~6--8 then test how far the same organizing logic can be carried: first to a vector-valued echo, then to construction and comparison of reductions, and finally to higher-order targets.

\section{The geometric target: \texorpdfstring{\(\Phi\) and \(\mu\)}{Phi and mu}}

Let \(U\subset\R^2\) be open and let \(\omega=M(x,y)\,dx+N(x,y)\,dy\) be a \(C^1\) one-form representing an explicit equation \(y'=f(x,y)\) via the normalized representative \(\omega=dy-f(x,y)\,dx\). Exactness means that \(\omega=dF\) for some potential \(F\). Hence \(M_y=N_x\); conversely, closedness is locally sufficient for exactness and is globally sufficient on a simply connected planar domain. The computational target is therefore a representative from whose coefficients a potential \(F\) can be recovered by ordinary integration.

We work locally on regular coordinate patches. A coordinate change \(\Phi:V\to U_0\subset U\) is a diffeomorphism between such patches, and solution curves satisfy \(\omega(\dot\gamma)=0\). If \(\mu\Phi^*\omega=dG\), then the transformed solutions are connected components of regular level sets of \(G\); applying \(\Phi\) recovers the corresponding solution curves in the original variables.

For the reductions considered here, two mathematically distinct operations are used to reach the target \(\mu\Phi^*\omega=dG\):

\paragraph{1. Rescaling (The Integrating Factor \(\mu\)):} 
Multiplying \(\omega\) by a non-vanishing function \(\mu\) preserves the solution foliation (\(\ker(\mu\omega)=\ker\omega\)) but may change its closedness, since \(d(\mu\omega) = d\mu \wedge \omega + \mu\,d\omega\). Figure~\ref{fig:foliation} illustrates this: for \(\omega=dy-y\,dx\), the factor \(e^{-x}\) does not change the unparameterized curves, but turns them into explicit level sets of \(F(x,y) = e^{-x}y\).

\begin{figure}[ht]
\centering
\begin{minipage}{0.45\textwidth}
\centering
\begin{tikzpicture}
\begin{axis}[
width=\linewidth,
height=5.0cm,
axis lines=middle,
xmin=-1.2,xmax=1.2,
ymin=-3,ymax=3,
xtick={-1,0,1},
ytick={-2,0,2},
title={\(\omega=dy-y\,dx\)},
samples=100,
clip=true]
\addplot[thick,domain=-1.2:1.2] {exp(x)};
\addplot[thick,domain=-1.2:1.2] {0.5*exp(x)};
\addplot[thick,domain=-1.2:1.2] {-0.5*exp(x)};
\addplot[thick,domain=-1.2:1.2] {-exp(x)};
\end{axis}
\end{tikzpicture}
\end{minipage}
\hfill
{\Large\(\Longrightarrow\)}
\hfill
\begin{minipage}{0.45\textwidth}
\centering
\begin{tikzpicture}
\begin{axis}[
width=\linewidth,
height=5.0cm,
axis lines=middle,
xmin=-1.2,xmax=1.2,
ymin=-3,ymax=3,
xtick={-1,0,1},
ytick={-2,0,2},
title={\(dF=e^{-x}\omega,\;F=e^{-x}y\)},
samples=100,
clip=true]
\addplot[thick,domain=-1.2:1.2] {exp(x)};
\addplot[thick,domain=-1.2:1.2] {0.5*exp(x)};
\addplot[thick,domain=-1.2:1.2] {-0.5*exp(x)};
\addplot[thick,domain=-1.2:1.2] {-exp(x)};
\node[anchor=west] at (axis cs:0.55,1.74) {\scriptsize \(F=1\)};
\node[anchor=west] at (axis cs:0.55,0.87) {\scriptsize \(F=\tfrac12\)};
\node[anchor=west] at (axis cs:0.55,-0.87) {\scriptsize \(F=-\tfrac12\)};
\node[anchor=west] at (axis cs:0.55,-1.74) {\scriptsize \(F=-1\)};
\end{axis}
\end{tikzpicture}
\end{minipage}
\caption{The nonvanishing multiplier \(e^{-x}\) preserves the solution line field. The same curves are now represented as level sets of the explicitly available first integral \(F=e^{-x}y\).}
\label{fig:foliation}
\end{figure}
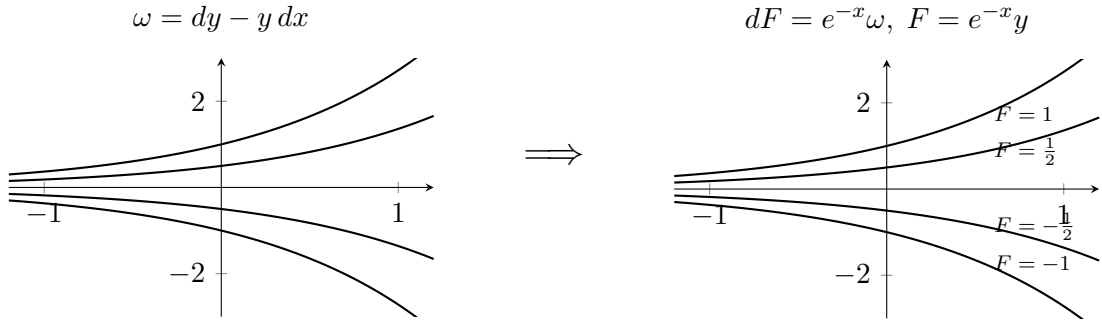

From the broader ODE perspective, an integrating factor is best viewed as a mechanism for producing a first integral and hence lowering the order. Schematically, write \(E[y]=0\) for a differential equation, let \(D_x\) denote total differentiation with respect to \(x\), and let \(I\) be an expression in \(x,y\) and, in higher order, derivatives of \(y\). If a multiplier converts the equation into a total derivative,
\[
\mu\,E[y]=D_x I,
\]
then every solution satisfies \(D_xI=0\), and one integration gives the first integral \(I=C\). In the first-order setting of this paper that reduction is terminal: there is no lower differential order left to solve. This is one reason exactness is such a natural endpoint here rather than merely another convenient form.

\paragraph{2. Coordinate Adaptation (The Pullback \(\Phi^*\)):}
A change of coordinates \(\Phi\) does not create exactness (\(d(\Phi^*\omega)=\Phi^*(d\omega)\)). Its structural role is strictly to expose a representation in which a useful multiplier \(\mu\) or factorization becomes algebraically accessible. 

For a smooth regular planar line field, a local integrating factor exists: equivalently, a smooth nowhere-vanishing one-form is locally proportional to the differential of a function. But finding such a factor generally requires solving \(\mu_x+f\mu_y=-f_y\mu\), whose characteristics satisfy the original ODE itself. Finding an unrestricted integrating factor is therefore not a shortcut. The recognizable solvable classes matter because they expose transformations for which the needed rescaling becomes accessible.

\section{The structural origin of separable and linear equations}

We know that, essentially, first-order equations are easier than higher-order equations, and linear equations are easier to solve than nonlinear ones. This section explains why linear and separable equations are first to appear. The answer is simple: solving them demands only one of the two core actions. 

\begin{proposition}[One-variable integrating factors]\label{prop:factors}
Let \(f\in C^1(I\times J)\), where \(I,J\) are open intervals, and let
\[
\omega=dy-f(x,y)\,dx.
\]
\begin{enumerate}
\renewcommand{\labelenumi}{(\roman{enumi})}
\item \(\omega\) admits a nowhere-vanishing integrating factor \(\mu=\mu(x)\) \textbf{if and only if} \(f(x,y)=A(x)y+B(x)\).
\item \(\omega\) admits a nowhere-vanishing integrating factor \(\mu=\mu(y)\) \textbf{if and only if} \(f(x,y)=a(x)b(y)\), where \(b\) is nowhere zero on \(J\).
\end{enumerate}
\end{proposition}

\begin{proof}
The integrating-factor PDE is \(\mu_x+f\mu_y=-f_y\mu\). 

If \(\mu=\mu(x)\), this reduces to \(\frac{\mu'}{\mu}=-f_y\). The left-hand side depends only on \(x\), forcing \(f_y=A(x)\), hence \(f(x,y)=A(x)y+B(x)\). Conversely, if \(f=A(x)y+B(x)\), then
\[
\mu(x)=\exp\!\left(-\int_{x_0}^{x}A(s)\,ds\right)
\]
is nowhere zero and satisfies the integrating-factor equation.

If \(\mu=\mu(y)\), the PDE is \(f\mu'+\mu f_y=0 \implies \frac{\partial}{\partial y}(\mu f)=0\), hence \(\mu(y)f(x,y)=a(x)\), so \(f=a(x)b(y)\) with \(b(y)=1/\mu(y)\). Conversely, if \(f=a(x)b(y)\) and \(b\neq0\) on \(J\), then \(\mu(y)=1/b(y)\) is a nowhere-vanishing integrating factor.
\end{proof}

This elementary characterization links the target directly to part of the familiar taxonomy: the linear and separable classes arise from the simplest coordinate restrictions on \(\mu\). Dependence of an integrating factor on only one coordinate is not an invariant property of the underlying foliation; it is a property relative to a chosen coordinate representation.

From this viewpoint, Bernoulli is a natural first nonlinear class beyond these one-step cases. In the generic Bernoulli reduction, a nonlinear coordinate change \(z=y^{1-n}\) is needed before the one-variable integrating-factor mechanism becomes available. It is therefore a standard early example in which the two structural operations are used in sequence.

\section{The two-floor architecture: A reduction hierarchy}

Instead of treating Exact, Separable, Linear, Homogeneous, and Bernoulli equations as distinct algebraic chapters, they are simply specific configurations of the ``master pair'' \((\Phi, \mu)\). Figures~\ref{fig:twofloors} and~\ref{fig:graph} illustrate this architecture.

\begin{figure}[ht]
\centering
\begin{tikzpicture}[
  box/.style={draw,rounded corners,align=center,inner sep=7pt,font=\small},
  arr/.style={-{Latex[length=2mm]},thick},
  both/.style={{Latex[length=2mm]}-{Latex[length=2mm]},thick}
]
\node[box,text width=0.85\linewidth] (struct) {
\textbf{Structural floor (structural view)}\\[2mm]
\(\omega\ \xrightarrow{\text{coordinate adaptation }\Phi}\ \Phi^*\omega
\ \xrightarrow{\text{nonvanishing rescaling }\mu}\ dG
\ \xrightarrow{\text{integration}}\ G=C\)};

\node[box,text width=0.85\linewidth,below=18mm of struct] (oper) {
\textbf{Operational floor (executable realization)}\\[2mm]
\(\begin{array}{c}
\text{homogeneous}\to\text{separable}\to\text{exact},\qquad
\text{Bernoulli}\to\text{linear}\to\text{exact},\\[1mm]
\text{Riccati}+y_p\to\text{linear}
\end{array}\)};

\draw[both] (struct.south) -- node[right,align=left] {recognize \& justify\\execute \& interpret} (oper.north);
\end{tikzpicture}
\caption{The two-floor reduction architecture. The upper floor records the target and the mathematical type of each move; the lower floor records the named classes and the executable reductions that realize those moves. AI or CAS may assist with selected operational steps, but recognition, justification, domain control, and interpretation remain mathematical tasks across both floors.}
\label{fig:twofloors}
\end{figure}
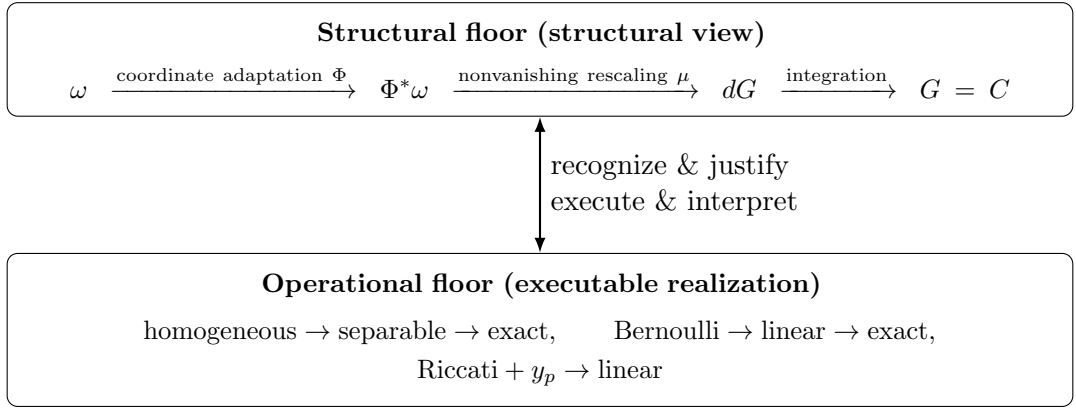

For the classes considered here, the reduction graph can be summarized by specifying the \((\Phi,\mu)\) data that lead to \(\mu\Phi^*\omega=dG\). When an arrow is displayed at the level of a normalized explicit equation, normalization of the coefficient of the new dependent differential is understood as part of the indicated nonvanishing rescaling.

\begin{itemize}
    \item \textbf{Exact Equations:} The terminal node. \(\Phi = \mathrm{id}\), \(\mu = 1\).
    \item \textbf{Separable (\(y'=a(x)b(y)\)):} Requires no pullback (\(\Phi = \mathrm{id}\)), but utilizes a dependent-variable rescaling \(\mu(y) = 1/b(y)\) (as dictated by Proposition~\ref{prop:factors}). \emph{Structural audit:} if \(b(c)=0\), division removes the equilibrium solution \(y\equiv c\), which must be restored.
    \item \textbf{Linear (\(y'+p(x)y=q(x)\)):} Requires no pullback (\(\Phi = \mathrm{id}\)), but utilizes an independent-variable rescaling \(\mu(x) = \exp(\int p(x)\,dx)\).
    \item \textbf{Homogeneous (\(y'=H(y/x)\)):} On \(x\neq0\), use \(\Phi(x,v)=(x,xv)\), where \(v=y/x\). Then
    \[
    \frac{1}{x[H(v)-v]}\,\Phi^*\omega
    =\frac{dv}{H(v)-v}-\frac{dx}{x},\qquad H(v)\neq v.
    \]
    Thus the coordinate change exposes a separable equation and the subsequent rescaling reaches exactness. Roots \(H(c)=c\) give the straight-line solutions \(y=cx\), which must be restored after division.
    \item \textbf{Bernoulli (\(y'+p(x)y=q(x)y^n\), \(n\neq1\)):} Work on a real branch where \(z=y^{1-n}\) is smooth and locally invertible, and let \(\Phi(x,z)=(x,\psi(z))\), where \(\psi\) is the local inverse. After the corresponding normalization, the equation reaches the Linear node, where an \(x\)-dependent integrating factor completes the reduction.
    \item \textbf{Riccati (with known \(y_p\)):} A translation and inversion pullback \(\Phi(x,u)=(x, y_p + 1/u)\) acts as a bridge directly to the Linear node.
\end{itemize}

\textbf{Structural note: excluded branches.} In elementary examples, many exceptional solutions---often called ``singular'' solutions in classroom practice---appear precisely on sets excluded by a division or by a noninvertible substitution. The structural floor makes the recovery of such lost branches part of the reduction itself rather than an afterthought. In this architecture, branch recovery is therefore not a supplementary warning attached to a solved problem; it is one of the mathematical obligations of the reduction.

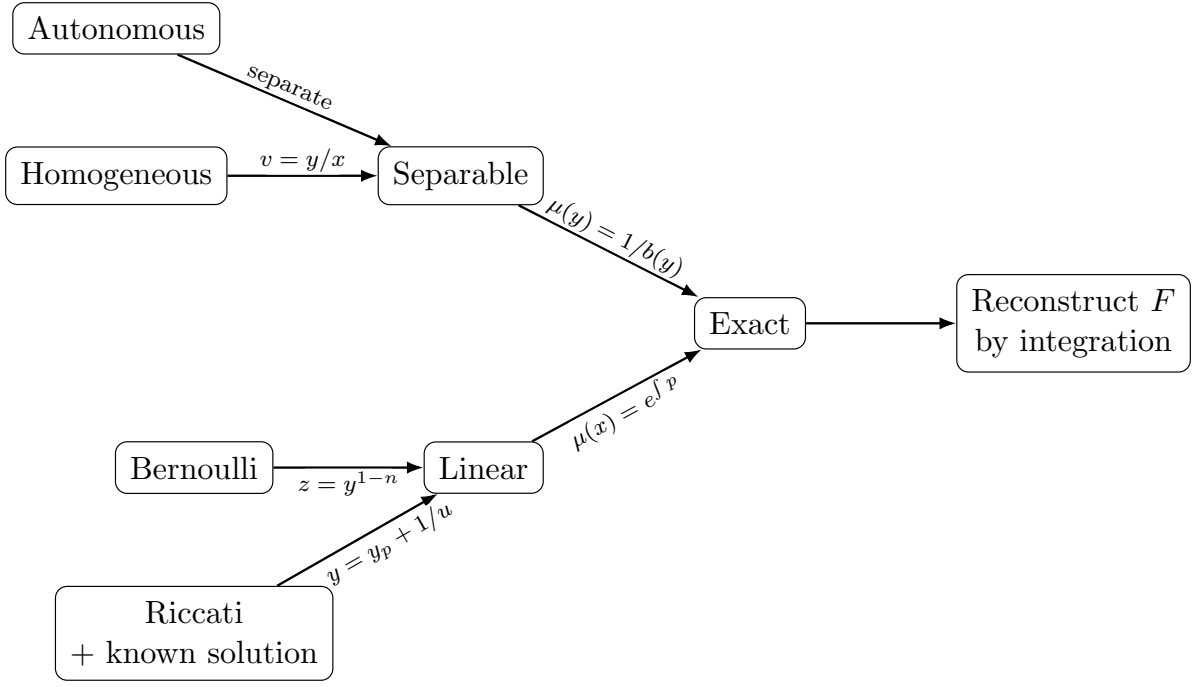
\begin{figure}[ht]
\centering
\resizebox{0.98\linewidth}{!}{%
\begin{tikzpicture}[
node distance=11mm and 18mm,
every node/.style={draw,rounded corners,align=center,inner sep=5pt},
arr/.style={-{Latex[length=2mm]},thick},
lab/.style={draw=none,fill=white,inner sep=1pt,font=\scriptsize}
]
\node (exact) {Exact};
\node (quad) [right=of exact] {Reconstruct \(F\)\\by integration};
\node (sep) [above left=of exact] {Separable};
\node (lin) [below left=of exact] {Linear};
\node (hom) [left=of sep] {Homogeneous};
\node (aut) [above=of hom] {Autonomous};
\node (ber) [left=of lin] {Bernoulli};
\node (ric) [below=of ber] {Riccati\\+\ known solution};

\draw[arr] (exact) -- (quad);
\draw[arr] (sep) -- node[lab,above,sloped] {\(\mu(y)=1/b(y)\)} (exact);
\draw[arr] (lin) -- node[lab,below,sloped] {\(\mu(x)=e^{\int p}\)} (exact);
\draw[arr] (hom) -- node[lab,above,sloped] {\(v=y/x\)} (sep);
\draw[arr] (aut) -- node[lab,above,sloped] {separate} (sep);
\draw[arr] (ber) -- node[lab,below,sloped] {\(z=y^{1-n}\)} (lin);
\draw[arr] (ric) -- node[lab,below,sloped,pos=0.64,yshift=-1pt] {\(y=y_p+1/u\)} (lin);
\end{tikzpicture}%
}
\caption{The condensed reduction graph. Every class displayed here is represented by a path to the Exactness node via coordinate adaptation and/or nonvanishing rescaling.}
\label{fig:graph}
\end{figure}

Figure~\ref{fig:graph} also displays a simple hierarchy of structural depth when read from the outer classes toward the Exact node. Homogeneous, Bernoulli, and Riccati equations (once the required particular solution is known) first need a coordinate or dependent-variable adaptation and then enter a one-step class; separable and linear equations are already one rescaling away from exactness; exact equations sit at the terminal first-order node, after which ordinary integration reconstructs the potential. The point is not to rank algebraic difficulty, but to count how many structural moves separate a given representation from the common target.

\subsection*{Worked example on both floors}
Consider the Bernoulli equation: \(y'+y=e^x y^2\).
On the \emph{operational floor}, the student applies the standard substitution \(z=y^{-1}\), obtains \(z'-z=-e^x\), finds the integrating factor \(e^{-x}\), and integrates. The substitution excludes \(y=0\), so the solution \(y\equiv0\) must be restored separately. 

On the \emph{structural floor}, the logic is transparent and unified. We have \(\omega=dy+(y-e^{x}y^2)\,dx\). 
The pullback \(\Phi(x,z)=(x,1/z)\) exposes the linear form. We then find the combined factor to force exactness:
\[
\boldsymbol{(-z^2e^{-x})\,\Phi^*\omega=d(e^{-x}z+x)=0.}
\]
The sequence of formulas is not merely a recipe attached to a name; it is an explicit realization of \(\omega \xrightarrow{\Phi^*} \Phi^*\omega \xrightarrow{\mu} dG\). 

\subsection*{From execution to structural assessment}
Once the two floors are visible, the same equation supports questions that go beyond naming a method and carrying it out. One may give a proposed coordinate change and ask whether the pullback actually exposes a simpler class; give a candidate multiplier and ask whether it makes the pulled-back form closed; perturb a coefficient and ask whether the same route still works; or delegate a routine integral to CAS or AI while requiring the student to recover excluded branches and verify the resulting first integral in the original variables. These tasks still involve calculation, but the calculation is located inside a common reduction architecture. The division of labor between human and machine may vary; the mathematical target does not.

The diagram can also be traversed backward. Starting from a simple potential \(G\), together with a chosen coordinate map \(\Phi\) and nonvanishing factor \(\mu\), the relation \(\mu\Phi^*\omega=dG\) generates an ODE whose reduction path is known by construction. In the Bernoulli example above, the data \(G=e^{-x}z+x\), \(\Phi(x,z)=(x,1/z)\), and \(\mu=-z^2e^{-x}\) recover the original nonlinear equation. This reverse direction supplies a direct way to design variations: change one piece of the triple \((G,\Phi,\mu)\) and ask which structural features survive.

\section{Lie symmetry: The source of \texorpdfstring{\(\Phi\)}{Phi}}

How can one discover a useful coordinate change \(\Phi\) when it is not algebraically obvious? Lie symmetry theory provides a systematic route when a suitable symmetry is available. This is also the viewpoint emphasized by Yap \cite{Yap2010}: a symmetry can guide us toward coordinates in which the equation becomes easier to solve.

Suppose a one-parameter symmetry generator \(X=\xi\partial_x+\eta\partial_y\) leaves the foliation invariant, so that \(\Lie_X\omega=a\omega\). Locally, where \(X\neq0\), one chooses canonical coordinates \((r,s)\) satisfying
\[
Xr=0,\qquad Xs=1,\qquad dr\wedge ds\neq0.
\]
Here \(r\) is constant along the symmetry orbits, while \(s\) measures motion along them; the condition \(dr\wedge ds\neq0\) says that \((r,s)\) are genuine local coordinates. In these coordinates the symmetry is straightened to \(X=\partial_s\). Thus the passage to \((r,s)\) is exactly a coordinate-adaptation step in our language.

For a transverse symmetry, the link with integrating factors is especially sharp. Suppose \(\Lie_X\omega=a\omega\) and \(\omega(X)\neq0\). Set
\[
 h=\omega(X),\qquad \alpha=\frac{\omega}{h}.
\]
Then \(Xh=(\Lie_X\omega)(X)=ah\), so \(\Lie_X\alpha=0\), while \(\alpha(X)=1\). Cartan's formula gives
\[
0=\Lie_X\alpha=\iota_X d\alpha+d(\alpha(X))=\iota_X d\alpha.
\]
In dimension two, since \(X\neq0\), this implies \(d\alpha=0\). Hence locally
\[
\boldsymbol{\mu=\frac{1}{\omega(X)}}
\]
is an integrating factor. 

For example, the homogeneous equation \(y'=H(y/x)\) is invariant under the dilation symmetry \(X=x\partial_x+y\partial_y\). The canonical coordinates are \(r=y/x\) and \(s=\log|x|\). The standard trick \(v=y/x\) is not magic; it is precisely a natural symmetry-adapted coordinate for the dilation action (Figure~\ref{fig:lie}).

\begin{figure}[ht]
\centering
\begin{minipage}{0.46\textwidth}
\centering
\begin{tikzpicture}[scale=1.05]
\draw[->] (-0.2,0) -- (3.6,0) node[right] {\(x\)};
\draw[->] (0,-0.2) -- (0,3.1) node[above] {\(y\)};
\foreach \m in {0.4,0.8,1.2}{
  \draw[dashed] (0,0) -- (2.45,{2.45*\m});
}
\draw[very thick,domain=0.55:3.05,samples=80]
  plot (\x,{\x*(0.55+0.20*ln(\x))});
\node[align=left] at (2.15,2.75) {\scriptsize dashed: dilation orbits\\[-1pt]\scriptsize solid: \(y=x(0.55+\tfrac15\log x)\)};
\end{tikzpicture}
\end{minipage}
\hfill
{\Large\(\longmapsto\)}
\hfill
\begin{minipage}{0.46\textwidth}
\centering
\begin{tikzpicture}[scale=1.05]
\draw[->] (-1.2,0) -- (2.5,0) node[right] {\(s=\log|x|\)};
\draw[->] (0,-0.2) -- (0,2.4) node[above] {\(r=y/x\)};
\foreach \r in {0.4,0.8,1.2}{
  \draw[dashed] (-1.1,\r) -- (2.25,\r);
}
\draw[very thick,domain=-0.8:2.0,samples=80]
  plot (\x,{0.55+0.20*\x});
\node[align=left,anchor=west] at (0.62,2.08) {\scriptsize \(X=\partial_s\)\\[-1pt]\scriptsize symmetry orbits: \(r=\text{const}\)};
\end{tikzpicture}
\end{minipage}
\caption{An explicit coordinate adaptation for \(y'=y/x+1/5\) on \(x>0\). The curved solution family \(y=x(0.55+\tfrac15\log x)\) is pulled back via \(\Phi\) to a family of straight lines \(r=0.55+s/5\).}
\label{fig:lie}
\end{figure}
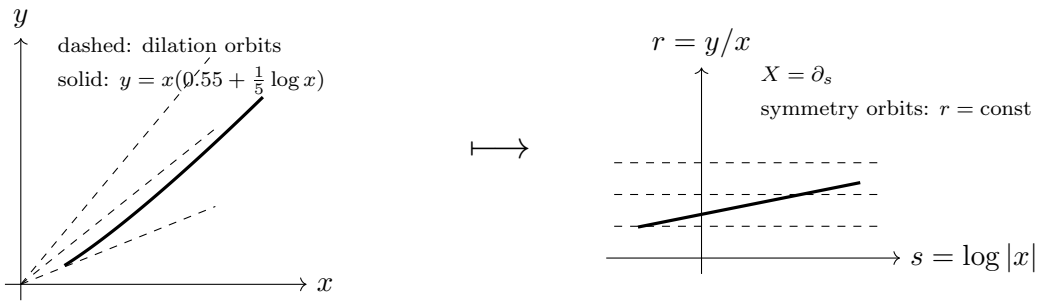
\FloatBarrier

\section{A vector-valued echo: Constant-coefficient systems}

The same reduction logic has a useful vector-valued echo. Let \(A\in\mathbb R^{n\times n}\) be constant and \(\mathbf g:I\to\mathbb R^n\) continuous. Consider
\[
\mathbf{x}'=A\mathbf{x}+\mathbf{g}(t).
\]
Introduce the time-dependent linear change of variables
\[
\boldsymbol{\mathbf{z}(t)=e^{-At}\mathbf{x}(t).}
\]
Differentiating this yields \(\mathbf{z}' = -Ae^{-At}\mathbf{x}+e^{-At}\mathbf{x}' = e^{-At}\mathbf{g}(t)\). 

The system has been transformed directly to the atomic vector-valued direct-integration state: \(d\mathbf{z}=e^{-At}\mathbf{g}(t)\,dt\). In the homogeneous case, \(\mathbf{z}'=0\). This is a vector-valued echo of the same reduction logic, not an application of the planar scalar integrating-factor theorem: the invertible matrix \(e^{-At}\) removes the homogeneous dynamics and leaves direct integration \cite{Boyce2017,Teschl2012}.

\section{From Taxonomy to Structural Assessment}

The architecture becomes pedagogically more consequential when it changes not only how familiar methods are explained, but also the kinds of problems one can pose. The operational surface of two problems may look completely different---homogeneous substitution, Bernoulli reduction, a hidden composite coordinate, or an integrating factor---while the structural grammar remains
\[
\boxed{\omega\ \longrightarrow\ \Phi^*\omega\ \longrightarrow\ \mu\Phi^*\omega=dG\ \longrightarrow\ G=C.}
\]
This makes it possible to design a progression in which students first manipulate the structural moves, then reinterpret the standard named classes through the same grammar, and finally confront equations for which a familiar classification, even when available, does not by itself reveal the most useful executable reduction.

\subsection*{Reverse construction}
The same diagram can be used in the opposite direction. Let \(\Phi:V\to U\) be a smooth diffeomorphism between coordinate patches, let \(\mu\) be a smooth nowhere-zero function on \(V\), and let \(G\) be smooth with \(dG\neq0\) on the regular patch under consideration. Define a one-form on \(U\) by
\[
\omega=(\Phi^{-1})^*\!\left(\frac{1}{\mu}\,dG\right).
\]
Then, by construction,
\[
\mu\Phi^*\omega=dG.
\]
Thus the data \((G,\Phi,\mu)\) generate a regular Pfaffian equation whose reduction path is known in advance. An explicit representation \(y'=f(x,y)\) requires, in addition, that the coefficient of \(dy\) be nonzero on the chosen patch.

The construction data are not unique, and this nonuniqueness is itself informative. With \(G\) and \(\Phi\) fixed, changing a nowhere-zero \(\mu\) changes the representative but not its solution foliation. Likewise, replacing \(G\) by \(h(G)\) and \(\mu\) by \(h'(G)\mu\), with \(h'\neq0\), leaves the constructed \(\omega\) unchanged. By contrast, changing the potential or the coordinate map can change the foliation or its presentation. The point is not that the pullback identity is new, but that within the present architecture it becomes a direct mechanism for designing examples and asking precisely what a change of representation preserves.

\subsection*{Comparing reduction paths}
A given equation may also admit more than one useful route. The relevant question is then not merely whether a substitution works, but what it exposes: does it produce a homogeneous or inhomogeneous linear equation, does it shorten the remaining calculation, and what domain restrictions does it introduce? Appendix~\ref{app:lab} includes a concrete comparison of this kind. It also develops a compact exercise progression involving reverse construction, lost branches, hidden coordinates, representation comparison, and a case in which the structural reduction is complete even though the remaining quadrature is not elementary. The sequence is expository rather than empirical: it demonstrates concretely the kinds of mathematical questions the architecture supports, while leaving learning outcomes and systematic discovery as separate questions.

\section{Beyond First Order: First Integrals and Structural Targets}

The point of this section is not to extend the formula \(\mu\Phi^*\omega=dG\) verbatim to every higher-order equation. It is to preserve the organizing logic that made the first-order picture useful: identify a privileged target, understand the structural move that reaches it, and audit what is gained or lost in the passage. In first order, exactness is terminal because the first integral \(G=C\) leaves no differential order to solve. In higher order, the analogous event is usually intermediate: a first integral produces a lower-order relation from which the reduction can continue.

Thus the first-order picture
\[
\boxed{\text{equation}\;\longrightarrow\;\text{exact representative}\;\longrightarrow\;G=C}
\]
becomes, on the order-reduction branch,
\[
\boxed{\text{higher-order equation}\;\longrightarrow\;\text{first integral}\;\longrightarrow\;\text{lower-order problem}}.
\]
The examples below are chosen to show this connection repeatedly rather than merely by analogy. Multiplication may expose a total derivative, a known solution may expose a first-order equation, and operator factorization may expose successive first-order stages. A second branch, represented by Sturm--Liouville normalization, has a different target: it does not lower the order, but moves the equation into a structural class from which another theory becomes available.

The point is not that every higher-order ODE admits such reductions, nor that there is a universal algorithm for finding them. When a first-integral relation
\[
I\bigl(x,y,\ldots,y^{(n-1)}\bigr)=C
\]
is used to solve explicitly for the highest remaining derivative, this is a local nondegenerate step; for example, one may require \(I_{y^{(n-1)}}\neq0\) on the branch under consideration. Otherwise the relation may still be useful, but singular branches require separate analysis.

\subsection*{A first integral as order reduction}
A familiar example already contains the whole idea. If
\[
y''=F(y),
\]
then multiplication by \(y'\) gives
\[
y'y''=F(y)y'
=\frac{d}{dx}\!\left(\int^y F(\eta)\,d\eta\right),
\]
so
\[
\frac{d}{dx}\left(\frac12(y')^2-\int^y F(\eta)\,d\eta\right)=0.
\]
The resulting first integral
\[
\frac12(y')^2-\int^y F(\eta)\,d\eta=C
\]
is a first-order relation. In elementary mechanics this is read as an energy integral; structurally, it is a one-step reduction of order. The connection with the first-order core is direct: a multiplier has converted the differential equation into a total derivative, and the computational target is again a conserved quantity. The difference is equally important---here the multiplier is part of an order-reduction step rather than a nowhere-vanishing rescaling that preserves a planar line field. Here the multiplier is \(y'\), which need not be nonzero. Every solution of the original equation satisfies the energy relation, but differentiating the relation recovers \(y''=F(y)\) directly only on branches where \(y'\neq0\). Turning points and constant candidates must therefore be checked in the original equation. This is different from the nowhere-vanishing rescalings used in the first-order core.

\subsection*{A known solution of a second-order linear equation}
Consider
\[
y''+p(x)y'+q(x)y=0
\]
and suppose a nowhere-vanishing solution \(y_1\) is known on the interval under consideration. Writing \(y=y_1v\) and using the equation satisfied by \(y_1\) yields
\[
v''+\left(2\frac{y_1'}{y_1}+p\right)v'=0.
\]
With \(w=v'\), the second-order problem has become the first-order linear equation
\[
w'+\left(2\frac{y_1'}{y_1}+p\right)w=0,
\]
whose integrating factor is
\[
y_1^2 e^{\int p(x)\,dx}.
\]
Thus the standard reduction-of-order construction is itself a two-stage reduction: a substitution exposes a first-order equation, and a multiplier completes that reduced step. In the language of the first-order architecture, \(y=y_1v\) plays the role of coordinate adaptation: it does not solve the equation, but exposes a representation in which a familiar one-variable integrating factor becomes available. Reconstruction is then explicit:
\[
v(x)=C_1+\int^x w(s)\,ds,
\]
so \(w\equiv0\) recovers the multiples of the known solution \(y_1\). The Wronskian/Abel formula is another expression of the same mechanism.

\subsection*{Constant coefficients as operator factorization}
For a constant-coefficient equation
\[
(D^2-(r_1+r_2)D+r_1r_2)y=0,
\qquad D=\frac{d}{dx},
\]
one has
\[
(D-r_1)(D-r_2)y=0.
\]
Setting \(z=(D-r_2)y\) produces
\[
(D-r_1)z=0,
\]
followed by the first-order reconstruction equation
\[
(D-r_2)y=z.
\]
The usual characteristic-polynomial method can therefore be read as a factorization of the second-order operator into first-order reduction stages. Here factorization plays the organizing role that coordinate adaptation played in first order: it exposes an intermediate variable governed by a simpler target equation, after which reconstruction returns to \(y\). Repeated and complex roots modify the algebra but not this structural reading.

\subsection*{General factorization and the Riccati bridge}
For a general monic second-order operator
\[
L=D^2+p(x)D+q(x),
\]
a factorization
\[
L=(D+a)(D+b)
\]
requires
\[
a+b=p,
\qquad
b'+ab=q.
\]
Eliminating \(a=p-b\) gives the Riccati equation
\[
b'+pb-b^2=q.
\]
Equivalently, if a nowhere-vanishing solution \(y_1\) of \(Ly=0\) is known and
\[
w=\frac{y_1'}{y_1},
\]
then
\[
w'+w^2+pw+q=0
\]
and
\[
\boxed{L=(D+p+w)(D-w).}
\]
This gives a precise sense in which ``knowing one solution'' and ``factoring the operator'' are two faces of the same order-reduction mechanism. The conceptual link to the main architecture is again target-centered: the factorization is useful because it manufactures an intermediate first-order object whose conserved quantity can be reached explicitly. The factorization also produces a first integral directly. If
\[
L=(D+a)(D+b),
\]
then
\[
\frac{d}{dx}\left[e^{\int a}(y'+by)\right]
=e^{\int a}Ly.
\]
Hence a solution of \(Ly=0\) first reaches the conserved quantity \(e^{\int a}(y'+by)=C\), followed by a first-order reconstruction equation. For the factorization generated by a known zero-free solution \(y_1\), this becomes
\[
I=e^{\int p}(y_1y'-y_1'y),
\qquad
I'=e^{\int p}y_1Ly.
\]
This identity makes the link among factorization, reduction of order, and the Wronskian first integral explicit.

\subsection*{Sturm--Liouville form as structural normalization}
A different kind of target appears, for real-valued coefficients \(P,Q,R\), in
\[
y''+P(x)y'+\bigl(Q(x)+\lambda R(x)\bigr)y=0.
\]
With
\[
\rho(x)=e^{\int P(x)\,dx},
\]
multiplication gives
\[
(\rho y')'+\rho\bigl(Q+\lambda R\bigr)y=0.
\]
Equivalently,
\[
-(\rho y')'-\rho Qy=\lambda\,\rho R\,y.
\]
After the usual relabeling of coefficients and signs, this is a Sturm--Liouville-type divergence form. The multiplier has not solved the equation; it has moved the differential expression into a structurally privileged, formally self-adjoint form. For a student coming from linear algebra, the analogy is useful: a Hermitian matrix is valuable not because its eigenvalue problem has already been solved, but because self-adjoint structure unlocks orthogonality and spectral decomposition. The Sturm--Liouville multiplier plays an analogous preparatory role in function space. It is not literally a change of basis, and it does not solve the ODE; rather, it exposes the self-adjoint structure to which spectral theory can be applied once the operator is properly realized.

A self-adjoint operator realization requires additional data and hypotheses---for example a suitable positive weight \(\rho R\) in the standard spectral setting, boundary conditions, and a specified function-space domain. With those qualifications, spectral and variational tools become available. This is the same target-centered philosophy with a different target: not order reduction, but a representation in which the next layer of theory becomes accessible.

\subsection*{Two kinds of higher-order target}
The first-integral branch of the extended picture can be summarized as
\[
E_n
\longrightarrow
I_{n-1}=C_1
\longrightarrow
I_{n-2}=C_2
\longrightarrow\cdots,
\]
whenever the relevant nondegenerate reductions exist. Here \(E_n[y]=0\) denotes an \(n\)-th order equation, while \(I_{n-1}(x,y,\ldots,y^{(n-1)})=C_1\) denotes a first integral involving derivatives only up to order \(n-1\). If that relation is nondegenerate in its highest derivative, it can be read locally as an \(n-1\)-st order problem; a second independent reduction may then produce \(I_{n-2}=C_2\), and so on. The diagram is therefore a schematic record of successive loss of differential order, not a claim that such a chain always exists or is algorithmically discoverable. For systems, several independent first integrals play the analogous role of confining trajectories to intersections of level sets, thereby reducing the accessible dimension rather than simply the scalar differential order.

Sturm--Liouville normalization belongs to an adjacent but different branch of the same target-centered viewpoint. It does not lower the order by producing a first integral; it moves the equation to a structurally privileged representation from which another body of theory becomes available. The present paper does not develop a general higher-order theory. It records these two continuations only to show how the first-order idea of a computational target can persist beyond the terminal exactness case.

\section{Connecting the dots: Pedagogy and prior work}

The individual ingredients of the preceding analysis---integrating factors, canonical coordinates, symmetry, and exact forms---have deep historical roots. Arnold \cite{Arnold1992} emphasized the geometric content of integration methods. Starrett \cite{Starrett2007}, Hydon \cite{Hydon2000}, Ibragimov \cite{Ibragimov1999}, and Olver \cite{Olver1993} have thoroughly mapped the connections between Lie groups, canonical coordinates, and classical reductions. Yap \cite{Yap2010} highlighted coordinate transformations as a unifying theme. Standard texts by Teschl \cite{Teschl2012} and Boyce, DiPrima, and Meade \cite{Boyce2017} detail the classical solvable classes. On the educational front, Ferzola \cite{Ferzola1994} described an early implementation of computer algebra in an undergraduate ODE course, and Rezvanifard et al.\ \cite{Rezvanifard2023} studied conceptual challenges surrounding exactness. 

Our contribution does not claim to invent these ingredients. The contribution is the particular target-centered architecture obtained by putting them together. Exact representatives form the terminal computational node; coordinate adaptation and rescaling have different roles; the standard named classes become executable paths to the target; and Proposition~\ref{prop:factors} explains why the linear and separable classes arise from the simplest coordinate restrictions on the integrating factor. The local target \(\mathbf{\mu\Phi^*\omega=dG}\) separates the \emph{mathematical intent} from the \emph{algebraic execution}.

Several of the cited sources anticipate particular activities that are brought together and developed here. Starrett uses coordinate changes to simplify opaque first-order equations, works backward from prescribed families of curves, and treats reduction to quadrature as a satisfactory endpoint \cite{Starrett2007}. Teschl's \emph{Graduate Studies in Mathematics} text gives a compact treatment of exact equations and integrating factors in its early explicit-solutions chapter: its exercises move between level-set descriptions and differential equations and explicitly note that unrestricted integrating-factor discovery is generally as hard as solving the original equation \cite{Teschl2012}. This places Teschl close to the present paper in mathematical family, but the emphasis is different. His discussion is one component of a broad ODE and dynamical-systems text; the present paper makes the exact representative the organizing target of a first-course computational architecture and separates the structural and operational floors throughout.

Rezvanifard et al. use puzzle tasks that ask students to audit cancellation and exactness and to reconcile different implicit descriptions of the same solution family \cite{Rezvanifard2023}. For second-order equations, Cheb-Terrab and Roche study integrating factors as total-derivative targets and derive families of equations from prescribed integrating-factor data \cite{ChebTerrab1999}. These are genuine precedents for individual moves. The present synthesis does not depend on denying that proximity; its claim is organizational. It combines target, coordinate adaptation, rescaling, construction, comparison, and auditing within one first-course reduction architecture.

\section{Conclusion: The role of the architect in the AI era}

The reduction graph concerns regular branches of symbolic first-order equations. It does not replace existence and uniqueness theory, numerical methods, or stability analysis. It is not a decision procedure for arbitrary first-order ODEs, and reduction to quadrature does not imply the existence of an elementary antiderivative. Within its scope, however, it reveals a clear hierarchy.

The architecture has a direct curricular implication in an AI-assisted environment. Selected routine symbolic steps may be delegated, while recognition, justification, domain control, verification, and interpretation become natural objects of instruction and assessment. One can ask, for example: \emph{What is the correct coordinate change? Why is it admissible? Which integrating factor completes the passage to an exact representative? What branches were lost during the transformation? How is the resulting computation verified in the original equation?} This is an expository and curricular proposal; questions of comparative learning outcomes are separate empirical questions.

Teach not only the formula attached to each solvable class, but the computational state to which that formula moves the equation. The resulting course asks students to act not merely as executors of symbolic routines, but as architects of the reduction.

The extended material develops the architecture in three directions without changing that conclusion. Reverse construction turns \((G,\Phi,\mu)\) into a problem-design mechanism; the structural laboratory uses the same grammar to compare representations and reduction paths and to audit what transformations preserve or exclude; and the higher-order outlook shows how first integrals, factorization, and structurally privileged forms continue the target-centered viewpoint beyond the terminal first-order case. The purpose is not to make the syllabus larger, but to make the continuity between its methods visible.

More broadly, this case study suggests a question that may be worth asking in other mathematically structured courses: whether an operational principle, supported by a clear theoretical backbone, can reorganize apparently separate techniques into a coherent mathematical practice. Any such reorganization would have to be justified from the mathematics of the particular subject; the present paper offers only one concrete instance of that possibility.

\appendix
\section{A Structural Exercise Laboratory: From Guided Moves to Unlabeled Equations}
\label{app:lab}

This appendix is designed as a proof of concept for the pedagogical claim of the paper. The exercises deliberately move through three stages. First, students manipulate coordinate changes, multipliers, and potentials before the usual method names carry the burden of recognition. Second, familiar classes are revisited on two floors: the operational calculation and the structural reason that the calculation works. Finally, the labels are removed and the surface forms become less transparent. The purpose of the last stage is to compare substantially different representations that nevertheless lead to the same kind of executable reduction,
\[
\boxed{\omega\xrightarrow{\Phi^*}\Phi^*\omega
\xrightarrow{\mu}dG\longrightarrow G=C.}
\]
The same checklist accompanies every surface form: Is the coordinate change locally admissible? Is the multiplier nonvanishing on the chosen branch? What was divided out or excluded? Does the proposed first integral differentiate back to the transformed equation? These questions belong to the structural floor even when the algebra on the operational floor is delegated.

\subsection{Stage I: Learn the moves before naming the method}

\paragraph{Exercise A.1 --- Constructing an equation from a simple conserved quantity.}
Work on the patch $x>0$. In coordinates $(x,v)$ let
\[
G(x,v)=\log x-\frac12v^2,
\qquad y=xv,
\qquad \mu=x^{-3}.
\]
Do not begin by classifying an ODE. Instead:
\begin{enumerate}[label=(\alph*)]
  \item Compute $dG$ and then the one-form $\widetilde\omega$ determined by $\mu\widetilde\omega=dG$.
  \item Rewrite $\widetilde\omega$ in the original variables $(x,y)$ using $v=y/x$ and $dy=v\,dx+x\,dv$.
  \item Show that the resulting form is
  \[
  \omega=(x^2+y^2)\,dx-xy\,dy.
  \]
  \item Now traverse the construction forward and recover the first integral.
\end{enumerate}

\begin{architectnote}[title={Structural note: backward traversal}]
The student has just traversed the reduction diagram backward before being asked to name what was done. A simple potential, a coordinate map, and a nonvanishing rescaling generate a surface form whose solution route is known by construction. This is the same architecture used later for solving, now used as a problem-design mechanism. On \(x>0\), the resulting Pfaffian form remains regular even at \(y=0\); writing it as an explicit equation \(y'=f(x,y)\) additionally requires \(y\neq0\), since the coefficient of \(dy\) vanishes there.
\end{architectnote}

\paragraph{Exercise A.2 --- What rescaling changes, and what it cannot change.}
For
\[
\omega=dy-y\,dx,
\qquad \mu(x)=e^{-x},
\]
verify directly that
\[
\mu\omega=d(e^{-x}y).
\]
Then answer two separate questions: why do $\omega=0$ and $\mu\omega=0$ have exactly the same unparameterized solution curves, and what new information becomes explicit once the representative is exact? Explain why multiplication or division by a nowhere-vanishing function preserves the solution curves on the common domain, and why excluding zeros of a divisor requires a separate check for omitted solutions.

\subsection{Stage II: Standard methods on two floors}

\paragraph{Exercise A.3 --- Homogeneous reduction with an actual lost branch.}
On $x\neq0$, consider
\[
y'=\left(\frac{y}{x}\right)^2+\frac{y}{x}-1.
\]
\begin{enumerate}[label=(\alph*)]
  \item Set $v=y/x$ and show that the transformed equation is
  \[xv'=v^2-1.\]
  \item On a branch where $v^2\neq1$, find a nonvanishing rescaling that makes the transformed one-form exact and obtain a first integral.
  \item Identify the branches removed by division and restore them in the original variables.
  \item State separately what happened on the structural floor and what happened on the operational floor.
\end{enumerate}

\paragraph{Exercise A.4 --- Bernoulli without treating the substitution as magic.}
Consider
\[
y'+y=e^x y^2.
\]
On a branch where $y\neq0$, use $z=y^{-1}$.
\begin{enumerate}[label=(\alph*)]
  \item Derive the pulled-back form, normalize the coefficient of $dz$, and obtain the resulting linear equation together with its $x$-dependent integrating factor.
  \item Combine the normalization and the linear integrating factor into a single nonvanishing rescaling relative to $\Phi^*\omega$, and exhibit $G$ such that $\mu\Phi^*\omega=dG$.
  \item Recover every solution of the original equation, including any branch excluded by the coordinate change.
  \item Explain which part of the argument would remain a mathematical obligation even if a CAS carried out all differentiation and integration correctly.
\end{enumerate}

\begin{architectnote}[title={One grammar, different operational surfaces}]
Exercises A.3 and A.4 are usually filed under different chapter headings. Here their operational paths differ, but the structural questions do not: choose an admissible coordinate, expose a simpler representative, rescale to an exact form, and audit what the reduction discarded.
\end{architectnote}

\subsection{Stage III: Unlabeled equations and representation comparison}

The next equations are not introduced by method name. The task is to search for a coordinate that compresses repeated algebraic structure. A successful choice should not merely simplify notation; after pullback it should expose a one-variable multiplier or an immediately separable representative. The examples are deliberately constructed and partly guided, so the emphasis is on exposing and comparing reductions rather than on a general discovery procedure.

\begin{challengebox}[title={Challenge A.5: a hidden diagonal coordinate}]
Solve
\[
dy+\bigl[1-x\bigl(1+(x+y)^2\bigr)\bigr]\,dx=0.
\]
Do not begin with a named-class test. Find a repeated functional block that can serve as a new coordinate. Justify local admissibility, perform the pullback, determine the rescaling, and produce a first integral. Then verify the result in the original variables.
\end{challengebox}

The structural route is exceptionally short once the coordinate is seen: with $u=x+y$,
\[
du-x(1+u^2)\,dx=0,
\qquad
\frac{du}{1+u^2}-x\,dx
=d\left(\arctan u-\frac{x^2}{2}\right).
\]
Thus the apparently nonstandard equation has the first integral
\[
\arctan(x+y)-\frac{x^2}{2}=C.
\]
Figure~\ref{fig:app-diagonal} shows the same family before and after the structural coordinate is exposed.

\begin{figure}[ht]
\centering
\begin{minipage}{0.47\textwidth}
\centering
\begin{tikzpicture}
\begin{axis}[
 width=\linewidth,height=5.0cm,
 axis lines=middle,xmin=-1.05,xmax=1.05,ymin=-2.2,ymax=2.2,
 samples=120,domain=-1:1,
 title={Original coordinates $(x,y)$},
 xtick={-1,0,1},ytick={-2,0,2},clip=true]
\addplot[thick] {tan(deg(-0.45+x^2/2))-x};
\addplot[thick,dashed] {tan(deg(x^2/2))-x};
\addplot[thick,densely dotted] {tan(deg(0.45+x^2/2))-x};
\end{axis}
\end{tikzpicture}
\end{minipage}
\hfill
\begin{minipage}{0.47\textwidth}
\centering
\begin{tikzpicture}
\begin{axis}[
 width=\linewidth,height=5.0cm,
 axis lines=middle,xmin=-1.05,xmax=1.05,ymin=-1.6,ymax=2.2,
 samples=120,domain=-1:1,
 title={Adapted coordinate $u=x+y$},
 xtick={-1,0,1},ytick={-1,0,1,2},clip=true]
\addplot[thick] {tan(deg(-0.45+x^2/2))};
\addplot[thick,dashed] {tan(deg(x^2/2))};
\addplot[thick,densely dotted] {tan(deg(0.45+x^2/2))};
\end{axis}
\end{tikzpicture}
\end{minipage}
\caption{A hard-looking surface form and the same solution family after the coordinate $u=x+y$ exposes the separable structure. The three curves correspond to three values of the same first integral $G=\arctan u-x^2/2$.}
\label{fig:app-diagonal}
\end{figure}
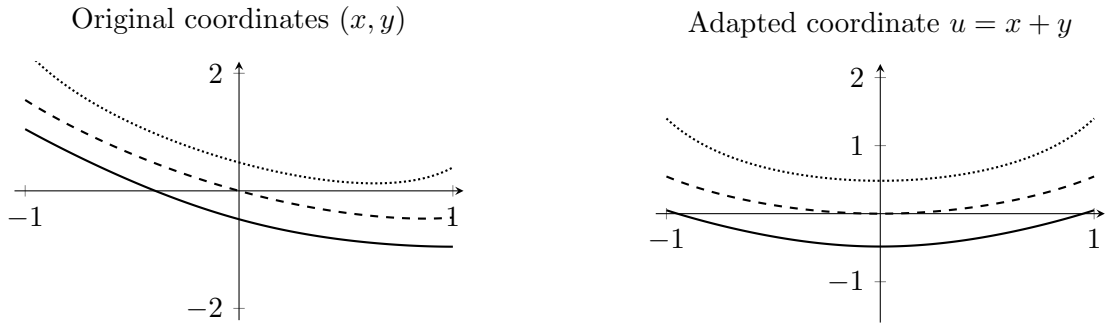

\begin{challengebox}[title={Challenge A.6: a hidden composite coordinate}]
Solve
\[
e^y\,dy+\bigl(2x-xe^y-x^3\bigr)\,dx=0.
\]
Find a single new dependent coordinate built from the repeated pieces of the one-form. Your coordinate should turn the equation into a first-order linear equation before any integrating factor is applied. Complete the reduction and write the solution as a level-set relation.
\end{challengebox}

A useful coordinate is $v=e^y+x^2$, for which
\[
dv=e^y\,dy+2x\,dx,
\qquad
dv-xv\,dx=0.
\]
The multiplier $e^{-x^2/2}$ then yields
\[
d\left(e^{-x^2/2}v\right)=0,
\qquad
\boxed{(e^y+x^2)e^{-x^2/2}=C.}
\]
Recovering \(y\) from \(v\) requires
\[
Ce^{x^2/2}>x^2,
\]
so that \(y=\log\bigl(Ce^{x^2/2}-x^2\bigr)\) is defined on the chosen real branch. The coordinate \(v\) exposes a homogeneous linear equation directly.

There is also a useful comparison of reduction paths. If one begins only with \(u=e^y\), then
\[
u'-xu=x^3-2x,
\]
whereas the adapted choice \(v=e^y+x^2\) gives immediately
\[
v'-xv=0.
\]
Here \(u_p=-x^2\) is a particular solution of the inhomogeneous equation for \(u\), and \(v=u-u_p\). Thus the second coordinate packages the familiar translation by a particular solution into the coordinate choice itself. Both routes are valid, but the second exposes more of the reduction at once; in a general linear equation, of course, finding such a particular solution may carry the original difficulty. Notice also that the reduced solution \(v\equiv0\) lies outside the image \(v=e^y+x^2>x^2\), so it does not produce an additional real solution of the original equation.

\begin{challengebox}[title={Challenge A.7: structural success without an elementary quadrature}]
On the patch $x>0$, $y>0$, solve structurally
\[
\frac{dy}{y}
+\left(x\log\frac{y}{x}-1-\frac1x\right)dx=0.
\]
Identify a natural dimensionless coordinate, reduce the equation to a standard first-order form, and stop only after a verified first integral has been produced. Do not treat a non-elementary antiderivative as a failure of the reduction.
\end{challengebox}

With $v=\log(y/x)$ one has $dv=dy/y-dx/x$, and the equation becomes
\[
dv+(xv-1)\,dx=0.
\]
Hence
\[
d\left(e^{x^2/2}v\right)-e^{x^2/2}\,dx=0,
\]
so a valid first integral is
\[
G(x,y)=e^{x^2/2}\log\frac{y}{x}
-\int_{x_0}^{x}e^{s^2/2}\,ds.
\]
The structural problem is finished although the remaining quadrature is not elementary.

\begin{challengebox}[title={Challenge A.8: a deliberately hostile surface form}]
On a coordinate patch avoiding $x=\pm1$ and $1+xy=0$, consider
\[
(1-x^2)\,dy
-\bigl(x^3y^2+x^3+4x^2y+xy^2+x+y^2-1\bigr)\,dx=0.
\]
Recognizing the Riccati form does not by itself supply the reduction exhibited below. Look for a fractional-linear coordinate built from $x+y$ and $1+xy$. Show that, after the correct coordinate adaptation and a nonvanishing rescaling, the equation collapses to the same separable target as Challenge A.5. Find the first integral and state the coordinate patch on which your argument is valid.
\end{challengebox}

Set
\[
v=\frac{x+y}{1+xy}.
\]
A direct calculation gives
\[
dv=\frac{(1-y^2)\,dx+(1-x^2)\,dy}{(1+xy)^2},
\]
and the displayed equation is equivalent, after division by the nonvanishing factor $(1+xy)^2$, to
\[
dv-x(1+v^2)\,dx=0.
\]
Therefore
\[
\boxed{
\arctan\!\left(\frac{x+y}{1+xy}\right)-\frac{x^2}{2}=C.}
\]
The local coordinate condition is $\partial v/\partial y=(1-x^2)/(1+xy)^2\neq0$, hence the stated restrictions. If the equations in Challenges A.5 and A.8 are solved explicitly for \(y'\), both can be placed in the Riccati family. That classification is compatible with the present point: the label describes the equation, while the adapted coordinate reveals the executable reduction. The pair therefore compares representations; it does not provide a general procedure for discovering such coordinates.

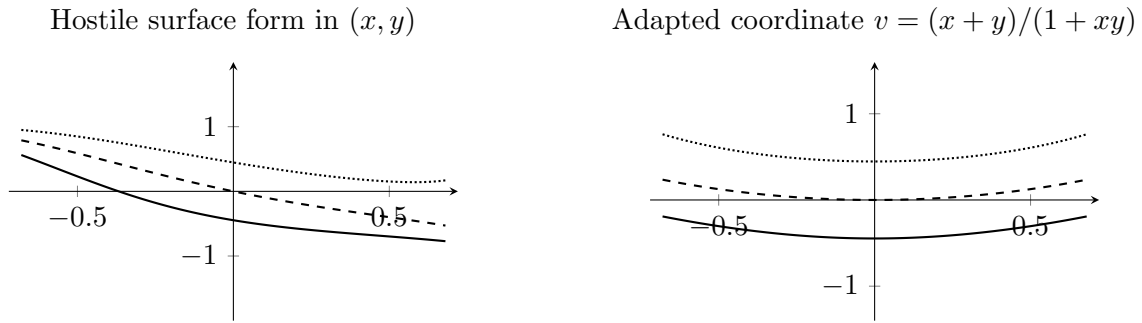
\begin{figure}[ht]
\centering
\begin{minipage}{0.47\textwidth}
\centering
\begin{tikzpicture}
\begin{axis}[
 width=\linewidth,height=5.0cm,
 axis lines=middle,xmin=-0.72,xmax=0.72,ymin=-2.0,ymax=2.0,
 samples=140,domain=-0.68:0.68,
 title={Hostile surface form in $(x,y)$},
 xtick={-0.5,0,0.5},ytick={-1,0,1},clip=true]
\addplot[thick] {(tan(deg(-0.42+x^2/2))-x)/(1-x*tan(deg(-0.42+x^2/2)))};
\addplot[thick,dashed] {(tan(deg(x^2/2))-x)/(1-x*tan(deg(x^2/2)))};
\addplot[thick,densely dotted] {(tan(deg(0.42+x^2/2))-x)/(1-x*tan(deg(0.42+x^2/2)))};
\end{axis}
\end{tikzpicture}
\end{minipage}
\hfill
\begin{minipage}{0.47\textwidth}
\centering
\begin{tikzpicture}
\begin{axis}[
 width=\linewidth,height=5.0cm,
 axis lines=middle,xmin=-0.72,xmax=0.72,ymin=-1.4,ymax=1.6,
 samples=140,domain=-0.68:0.68,
 title={Adapted coordinate $v=(x+y)/(1+xy)$},
 xtick={-0.5,0,0.5},ytick={-1,0,1},clip=true]
\addplot[thick] {tan(deg(-0.42+x^2/2))};
\addplot[thick,dashed] {tan(deg(x^2/2))};
\addplot[thick,densely dotted] {tan(deg(0.42+x^2/2))};
\end{axis}
\end{tikzpicture}
\end{minipage}
\caption{The visual point of coordinate adaptation. In the original variables the level curves are algebraically opaque; the fractional-linear coordinate reveals the elementary separable family $\arctan v-x^2/2=C$. The operational appearance changes drastically while the structural target is unchanged.}
\label{fig:app-mobius}
\end{figure}
\FloatBarrier

\subsection{What the sequence is intended to reveal}

The sequence is arranged so that the surface methods become progressively less useful as labels while the structural questions remain stable. Exercise A.1 teaches the diagram backward; Exercises A.3--A.4 connect the two floors on familiar territory; Challenge A.6 compares two valid reduction paths; Challenge A.7 separates structural completion from elementary closed form; and the pair A.5--A.8 shows that substantially different algebraic representations can lead to the same reduced equation. Figures~\ref{fig:app-diagonal} and~\ref{fig:app-mobius} make the last point geometrically: a coordinate change reorganizes an opaque family into one for which a simple rescaling exposes the potential.

In an AI-assisted setting, the same sequence also suggests a division of labor that is mathematically testable rather than rhetorical. A system may differentiate a proposed coordinate, simplify a pullback, integrate a one-variable expression, or search among candidate substitutions. The student can still be required to decide why a coordinate is admissible, why a multiplier is nonvanishing on the chosen branch, whether an excluded set contains genuine solutions, and whether the final $G$ differentiates back to the original equation. Thus the common heading becomes operational: very different first-order methods are treated as instances of the same repeated act of structural reduction.

\section*{Declaration of generative AI use}
After the substantive content was fully authored, OpenAI ChatGPT  tools were utilized to enhance readability, suggest structural flow improvements, improve language, proofread the manuscript, assist with LaTeX and code preparation, and conduct literature and source searches concerning prior work and originality. All AI-assisted material was independently reviewed and, where necessary, revised by the author.

\end{document}